\documentclass[1 [leqno,11pt]{amsart}
\usepackage{amssymb, amsmath,amsmath,latexsym,amssymb,amsfonts,amsbsy, amsthm}

\numberwithin{equation}{section}

\allowdisplaybreaks

\let\al=\alpha

\def\R{\mathbf R}

\newcommand{\beq}{\begin{equation}}
\newcommand{\eeq}{\end{equation}}
\newcommand{\ben}{\begin{eqnarray}}
\newcommand{\een}{\end{eqnarray}}
\newcommand{\beno}{\begin{eqnarray*}}
\newcommand{\eeno}{\end{eqnarray*}}

\newtheorem{theorem}{Theorem}[section]

\newtheorem{lemma}[theorem]{Lemma}
\newtheorem{proposition}[theorem]{Proposition}

\newtheorem{remark}[theorem]{Remark}

\begin{document}

\title[Stability of  Prandtl equation]
{Dynamic Stability for Steady Prandtl solutions}

\author{YAN GUO}
\address{Division of Applied Mathematics, Brown University, Providence, RI, 02906, USA
}
\email{yan\underline{ }guo@brown.edu}

\author{Yue Wang}
\address{School of Mathematical Sciences, Capital Normal University, 100048, Beijing, China}
\email{yuewang37@pku.edu.cn}

\author{Zhifei Zhang}
\address{School of Mathematical Sciences, Peking University, 100871, Beijing, China}
\email{zfzhang@math.pku.edu.cn}

\begin{abstract}
By establishing an invariant set \eqref{keying1B} for the Prandtl
equation in Crocco transformation, we prove orbital and asymptotic
stability of Blasius-like steady
states against Oleinik's monotone solutions.

\end{abstract}

\maketitle

\date{\today}
\section{Introduction}
In this paper, we study the Prandtl equation with outer flow $U\equiv 1:$
\begin{equation}\label{eq:Prandtl}
  \left\{
  \begin{aligned}
    &\partial_t u+u \partial_x u +v\partial_{y}u-\partial_{y}^2u=0,\quad in \quad  \Omega,\\
    &\partial_xu+\partial_y v=0,\quad in \quad  \Omega,\\
    &u|_{y=0}=v|_{y=0}=0\quad\mbox{and}\quad \displaystyle\lim_{y\to+\infty} u(t,x,y)=1,\\
     &u|_{t=0}= u_0, \quad u|_{x=0}= u_1,
  \end{aligned}
  \right.
\end{equation}where  $$\Omega=(0,T]\times(0,X]\times(0,+\infty).$$ Here we note $\partial_x p\equiv0$ because of $U\equiv 1$ and Bernoulli's law\begin{align}\label{Blaw}
\partial_tU+U\partial_x U+\partial_x p=0.
\end{align}

To describe the inviscid limit of the Navier-Stokes equation,
Prandtl\cite{Pr} developed the celebrated boundary layer theory by deriving the steady boundary layer equation \eqref{eq:Prandtl}. Soon after, Blasius\cite{Blasius} constructed the self-similar steady  solutions with the form:
\begin{align}\label{Blasiusu}
 [u_B,v_B]=\Big[f'(\zeta),\frac{1}{2\sqrt{x+x_0}}\{\zeta f'(\zeta)-f(\zeta)\}\Big],
\end{align}
where $\zeta=\frac{ y}{\sqrt{x+x_0}}$ with $x_0>0$ as a free parameter,
$\frac{1}{2}ff''+f'''=0,$ $f(0)=f'(0)=0$ and
$$1-f'(\zeta)\sim \zeta^{-1}e^{-\zeta^2C_1-C_2\zeta},\quad f''(\zeta)\sim \zeta(1-f')\sim e^{-\zeta^2C_1-C_2\zeta},\quad as\quad\zeta\rightarrow +\infty.
$$
Blasius solutions have  been experimentally confirmed with remarkable
accuracy as the main validation of the Prandtl theory (see \cite{Sch00} for instance) and
 Blasius solutions are spatial  limit of the
steady state solutions by the classical result of Serrin \cite{Serrin},
\begin{align}\label{Serrinconvergence}
\|u(x,y)-u_B(x,y)\|_{L^\infty_y}\rightarrow 0\quad \text{as}\quad x\rightarrow+\infty,
\end{align} and a refined description of the above asymptotics was established in \cite{I1} as well as in \cite{asyBgeneral}.

The goal of this paper is to establish dynamical stability of Blasius
type of steady state solutions for natural monotone perturbations with
$\partial_yu>0,$ for which global in time existence in Sobolev space is
guaranteed. Under $\partial_yu>0,$ it is standard to use Crocco
transformation
\begin{align*}
    \tau=t,\  \xi=x,\  \eta=u(t,x,y),\  w(\tau,\xi,\eta)=\partial_yu(t,x,y),\end{align*}
and   \eqref{eq:Prandtl} is transformed into
\begin{equation}\label{eq:Pran-Cro}
\left\{\begin{aligned}
&-\partial_\tau  w-\eta \partial_\xi w+w^2\partial_{\eta }^2 w=0,\quad (\tau,\xi,\eta)\in D,\\
&w\partial_\eta w\mid _{\eta=0}=0,\quad w\mid _{\eta=1}=0,
\\
&w\mid _{\tau=0}=w_0,\quad w\mid _{\xi=0}=w_1,
\end{aligned}\right.
\end{equation}
where $$D=(0,T]\times(0,X]\times(0,1).$$

Throughout this paper, we only consider the  following global \textit{Oleinik monotone solutions}:
 $
 w\in   C^2([0,+\infty)\times[0,X] \times[0,1))\cap C^3((0,+\infty)\times(0,X] \times(0,1)),
$
\begin{align}\label{eta1xi}
\displaystyle\lim_{\eta\to 1} \partial_\xi w=\displaystyle\lim_{\eta\to 1} \partial_\tau w=0,
\end{align}
and  for some continuous function $l(t,x),$ it holds
\begin{align}\label{ww-10uB}
  \displaystyle\lim_{\eta\to 1} w\partial_\eta^2 w
=l(t,x).
\end{align}
\smallskip

Since $w|_{\eta= 1} =0$, \eqref{eta1xi} is a natural condition. Instead of the classical approach of spectral analysis, our stability analysis is   based on the  discovery of    the   invariant set \eqref{keying1B} to the Prandtl equation \eqref{eq:Pran-Cro}.
    \begin{theorem} \label{w1tB}Let $w$ be the Oleinik monotone solution  to \eqref{eq:Pran-Cro} in $D$.  If
\begin{align}\label{ulbwB}\begin{split}
 &c_0(1-\eta)\leq w_1\leq c_0^{-1}(1-\eta)\sqrt{-\ln({\mu(1-\eta))}},\\&c_0(1-\eta)\leq w_0\leq c_0^{-1}(1-\eta)\sqrt{-\ln({\mu(1-\eta))}},\end{split}
\end{align} for some  constants $c_0\in(0,1)$, $\mu\in (0,\frac{1}{100}),$ and for some  constant $\alpha_0\in(0,1)$ and  some positive constant  $C_1$,
\begin{align}\label{epsilon1B}\begin{split}
&-C_1(1-\eta)^{\alpha_0}\leq\partial_\xi w+\partial_\tau w\leq \frac{(\delta b)^{\frac{1}{\alpha_0}}}{K}(1-\eta)\quad on\quad \xi=0 \quad and \quad \tau=0,\\& -C_1(1-\eta)^{\alpha_0}\leq\partial_\tau w \leq b\delta(1-\eta)^{\alpha_0}\quad on\quad \xi=0 \quad and \quad \tau=0,
\end{split}
\end{align}
where
 \begin{align}\label{bnumber}
 b=\frac{c_0}{4}e^{-\beta X}e^{-\frac{8}{3}}
\end{align}
with  $\beta$ depending only on   $\mu, c_0$ and $X$ but independent of $T$, $K=e^X\cdot \max\{2,2C_1^{\frac{1}{\alpha_0}-1}\}$ and $\delta=\min\{\frac{1}{10}\min_{\alpha\in[\frac{\alpha_0}{2},\alpha_0]}\alpha(1-\alpha)b^2,\frac{1}{321e^{X}},\frac{\min\{1,C_1\}}{b2^{\alpha_0}}\}$.
Then
\begin{align}\label{keying1B}\begin{split}
& b(1-\eta)\leq w\leq  c_0^{-1}(1-\eta)\sqrt{-\ln(\mu(1-\eta))}, \quad -C_1(1-\eta)^{\alpha_0}\leq\partial_\xi w+\partial_\tau w\leq \frac{\delta}{2}w,
\\&  w\partial_\eta^2 w\leq 2\delta , \quad -C_1(1-\eta)^{\alpha_0}\leq \partial_\tau w\leq  b\delta(1-\eta)^{\alpha_0}\quad in \quad D.\end{split}
\end{align}
\end{theorem}

We remark that the invariant set \eqref{keying1B} is natural for perturbations of the Blasius type profile, since $w_B=\partial_y u_B$, where we recall the Blasius solution $u_B$ in \eqref{Blasiusu},
satisfies  \begin{align}\label{xiwb}\begin{split}
 &\partial_\tau w_B=0,\quad\displaystyle\lim_{\eta\to 1} w_B\partial_\eta^2 w_B=-\frac{1}{2}\frac{1}{x+x_0},
\quad\partial_\xi w_B|_{\eta=0}\leq -a_0,\quad \partial_{\eta }^2 w_B|_{\eta=0}=0,\\&c(1-\eta)\leq w_B\leq  c^{-1}(1-\eta)\sqrt{-\ln(\mu(1-\eta))},\quad \partial_\xi w_B<0 \quad in\quad [0,X]\times[0,1),
 \end{split}
\end{align} for some positive constants $a_0,c,\mu,$  and in particular, for any constant $\alpha\in(0,1)$,
\begin{align}\label{weightsign}
&-C_{\alpha}(1-\eta)^{\alpha-1}<\frac{\partial_\xi w_B}{1-\eta} \sim\frac{w_B^2\partial_{\eta }^2 w_B}{1-\eta}\sim  -\frac{y}{(x+x_0)^{2}}\,\,as \,\,y\rightarrow +\infty\,\,or\,\, \eta\rightarrow 1,
\end{align} uniformly for $x\in[0,X].$

Moreover, $\partial_\xi w+\partial_\tau w\leq \frac{(\delta b)^{\frac{1}{\alpha_0}}}{K}(1-\eta)$ on $\tau=0$ and $\xi=0$ in \eqref{epsilon1B} can be replaced by an even more general condition  $\partial_\xi w+\partial_\tau w\leq \varepsilon_0w$ for some small positive constant $\varepsilon_0.$ However, we keep the present one since its proof is  brief and the method is the same.

\subsection{Main Result}
The following theorem is our main result on dynamic stability.
\begin{theorem}\label{snearb} Let $\bar{u}$ be a steady Oleinik monotone solution to \eqref{eq:Prandtl} and $u$  be a global Oleinik monotone solution such that
\begin{align}\label{negw2}
\partial_\eta^2 \bar{w}\leq 0,
\end{align}
and $(\bar{u}_0,\bar{u}_1),(u_0,u_1)$ satisfy \eqref{ulbwB}-\eqref{epsilon1B}.  Then we have the following results.

(i) Orbital stability: For any positive constant $\epsilon$, there exists a positive constant $\delta_\epsilon$, depending only on $\epsilon$, $\mu,\alpha_0, c_0,$  $C_1$ and $X$,  such that if \begin{align}\label{difutilde}
\|\partial_y u_1- \partial_y \bar{u}_1\|_{L^\infty([0,+\infty)\times[0,+\infty))}\leq \delta_\epsilon,\quad \|\partial_y u_0- \partial_y \bar{u}_0\|_{L^\infty([0,X]\times[0,+\infty))}\leq \delta_\epsilon,
\end{align}  then
\begin{align}\label{orbitalst}
\|\bar{u}- u\|_{L^\infty([0,+\infty)\times[0,X]\times[0,+\infty))}\leq \epsilon.
\end{align}

 (ii) Asymptotic stability: Furthermore, if for  a positive constant
      $\beta_0\in(0,
       b^2\alpha_0(1-\alpha_0)),
$
\begin{align}
|\partial_y u_1\big(t,u_1^{-1}(t,\eta)\big)- \partial_y\bar{u}_1\big(t,\bar{u}_1^{-1}(t,\eta)\big)|\leq e^{-t\beta_0}C_2(1-\eta)^{\alpha_0},
\end{align} where $u_1^{-1}(t,\eta)$ is the inverse function of $u_1(t,y)=\eta,$   or equivalently,
\begin{align}\label{pyupyub}
|w_1-\bar{w}_1|\leq e^{-\tau\beta_0}C_2(1-\eta)^{\alpha_0},
\end{align}  then for any $y_0\in \R_+$,
 there exists a positive constant $C$ depending on $y_0$ such that
\begin{align}\label{anyy0asmpst}
\|u-\bar{u}\|_{L^\infty([0,X]\times[0,y_0])}\leq C e^{-t\beta_0}\quad for \quad t\in [0,+\infty).
\end{align}Moreover,
\begin{align}\label{Raspst}
   \| u-\bar{u}\|_{L^\infty ([0,X]\times [0,+\infty))}\rightarrow 0, \quad as\quad  t\rightarrow+\infty.
\end{align}
\end{theorem}
Clearly from \eqref{xiwb}, \eqref{negw2} is a natural condition for Blasius-like
profile.
We remark that $X$ in our theorem is any positive constant. In light of \cite{Olei}'s global-in-$t$ existence of solutions on $[0,X]$ for small $X$ and \cite{XZ,XZZ}'s global-in-$t$ existence of solutions on $[0,X]$ for any positive constant $X,$ our result has no restriction on the size of $X$.
In addition,
$\partial_\xi w$ and $\partial_\tau w$ in \eqref{epsilon1B} are defined by $w_0$ and $w_1$ through the equation on the boundaries.

\subsection{Methodology of Proof}

We first estimate $w-\bar w$ which satisfies:
 \begin{align}\label{J(h)}
J(w-\bar{w})=-\partial_\tau  (w-\bar{w})-\eta \partial_\xi (w-\bar{w})+\big((w+\bar{w})\partial_{\eta }^2 \bar{w}\big)(w-\bar{w})+w^2\partial_{\eta }^2(w-\bar{w}).
\end{align}
We note if $w+\bar{w}\geq 0 $ and $\partial_\eta^2 \bar{w}\leq 0,$
then $J$ has a maximum principle  so that the orbital stability for $\bar{w}$: $|w- \bar{w}|\leq\varepsilon$  follows immediately. Such a crucial sign condition is satisfied by the famous Blasius profile.  For the asymptotic stability $|w- \bar{w}|\leq M e^{-\beta_0\tau},$  a \textit{uniform lower bound} is additionally employed to control the bad terms coming out of temporal function $e^{-\beta_0\tau}.$

However, it is challenging to control the original
$u-\bar u$ in terms of $w-\bar w$
 due to
the complex nature of nonlinear Crocco transformation which depends on the
unknown $\partial_yu$ itself. In fact, we have $$y=\int_0^{\bar{u}(y)}\frac{d\eta}{\bar{w}(\tau,\xi,\eta)}=
\int_0^{u(y)}\frac{d\eta}{w(\tau,\xi,\eta)}$$ so that
\begin{align*}
\int_{u(y)}^{\bar{u}(y)}\frac{d\eta}{\bar{w}}=
\int_0^{u(y)}\frac{\bar{w}-w}{\bar{w}w}d\eta.
\end{align*}
Hence, the uniform bound of $|u-\bar{u}|$ depends on the  bound of $|w-\bar{w}|$ and a \textit{uniform lower bound} of $w(\tau,\xi,\eta)$ which is independent of time $\tau$. We recall the classical time depending lower bound as \begin{align}\label{tempbarr}
w\geq ce^{-K\tau}(1-\eta).
\end{align} See Lemma \ref{prop:mono1B}.

The main technical novelty
of our method is to  establish a \textit{time independent} lower bound
  \begin{align}\label{introlbd}
  w\geq
b(1-\eta)
\end{align}as \textit{a part of} invariant set \eqref{keying1B},
where we recall $b$ in \eqref{bnumber}. To illustrate our method, we consider a smaller and simpler
invariant set (see \eqref{keying1B}),
\begin{align}\label{signinv}
 \quad\partial_\tau w\leq0 , \quad \partial_\xi w\leq 0\quad and \quad w\partial_\eta^2 w\leq 0,
\end{align} which leads to \eqref{introlbd}. Then
we have \eqref{introlbd} via an exponential barrier function in $\xi$
provided $\partial_\eta w\leq 0$ as  a consequence of \eqref{signinv}.

The proof of \eqref{signinv} is based on $\tau$ and $\xi$ invariance of our problem. Since there
is no boundary condition at $\xi=X,$  $\partial_\xi w$ and  $\partial_\tau w$ naturally satisfy the linearized Prandtl equation at $w:$
$$L(\partial_{\tau,\xi} w)=(2w\partial_{\eta }^2w)(\partial_{\tau,\xi} w)-\partial_\tau (\partial_{\tau,\xi} w) -\eta \partial_\xi(\partial_{\tau,\xi} w) +w^2\partial_{\eta }^2 (\partial_{\tau,\xi} w),$$ where $\partial_{\tau,\xi} w$ stands for $\partial_\xi w$ or  $\partial_\tau w.$
We observe that the linearized operator $L$ with the form\begin{align}\label{defl} L\cdot=(2w\partial_{\eta }^2w)\cdot-\partial_\tau \cdot -\eta \partial_\xi\cdot +w^2\partial_{\eta }^2 \cdot,\end{align} has a maximum principle as long as $w\partial_\eta^2 w\leq 0.$ We then conclude \eqref{signinv} via the following bootstrap scheme:
\begin{align}\label{signinv-121}
  w\partial_\eta^2 w\leq 0 \Rightarrow  \partial_\tau w\leq0  \,\,and  \,\, \partial_\xi w\leq 0 \,\,  \Rightarrow w^2\partial_{\eta }^2 w=\partial_\tau  w+\eta \partial_\xi w\leq 0.
\end{align}

Even though  the sign conditions are valid for perturbation  for any fixed $\eta<1$ by \eqref{xiwb}, \eqref{weightsign} for Blasius profile, our invariant set \eqref{keying1B} is more general, allowing positive tail for $\eta$ near $1$.

We establish \eqref{introlbd} via a bootstrap argument via  the  maximum principle based on the following proposition. See Lemma \ref{unilbd1}.

 \begin{proposition}\label{eta2negVIPB}
Let $w$ be the Oleinik monotone solution to \eqref{eq:Pran-Cro} in $D$  satisfying the conditions in Theorem \ref{w1tB}. Assume, for some positive  constant $T_1\in(0,T],$
\begin{align}\label{wb1-etaB-1}
 w\geq b(1-\eta)\quad in \quad [0,T_1]\times[0,X]\times[0,1].
\end{align}

Then
\begin{align}\label{eta2nB}
    w\partial_\eta^2 w\leq 2\delta\quad in \quad [0,T_1]\times[0,X]\times[0,1],
\end{align}
and \begin{align}\label{gn1B}\begin{split}
   & -C_1(1-\eta)^{\alpha_0}\leq \partial_\xi w+\partial_\tau w\leq \frac{\delta}{2}w\quad in \quad [0,T_1]\times[0,X]\times[0,1] ,\\ & -C_1(1-\eta)^{\alpha_0}\leq\partial_\tau w \leq b\delta(1-\eta)^{\alpha_0}\quad in \quad [0,T_1]\times[0,X]\times[0,1]. \end{split}\end{align}
\end{proposition}

We use another bootstrap argument for \eqref{eta2nB}. Assuming \eqref{eta2nB}, we first prove
\begin{align}\label{torefine}
   \partial_\xi w+\partial_\tau w\leq \frac{(\delta b)^{\frac{1}{\alpha_0}}}{K}(1-\eta)^{\alpha_0},\quad \partial_\tau w\leq b\delta(1-\eta)^{\alpha_0} .
\end{align}
 Since $\partial_{\eta }^2w$ are allowed to take positive values,  $L$ in \eqref{defl} is not suitable for applying the maximum principle. To overcome this difficulty, instead of constructing barrier functions for $\partial_{\tau,\xi} w$ based on $L\partial_{\tau,\xi} w=0$, we introduce a new function
\begin{align}\label{convexv}
g=\frac{\partial_{\tau,\xi} w}{v}
\end{align} for some concave function $v.$ The concavity of $v$ yields good terms to reconcile the bad effect from positivity of $2w\partial_{\eta }^2w$ when we consider $L_0g$ for   the operator $L_0$ with the form
\begin{align}\label{defl0}
 L_0=-\partial_\tau-\eta\partial_\xi+w^2\partial_\eta^2.
\end{align} See Lemma \ref{yuanstep1}.

Next, we improve \eqref{torefine} to \eqref{gn1B} by applying the maximum principle on \begin{align}
g=(\partial_\xi w+\partial_\tau w)-e^{-X}\frac{\delta}{2}w e^{\xi}.
\end{align} The key observation is  that if $ g $ attains a positive maximum at an interior point $z_{max},$ then $0<(\partial_\xi w+\partial_\tau w)(z_{max})\leq 6\delta w(z_{max})$ by $w\partial_\eta^2 w\leq 2\delta$ so that $$ L_0 (\partial_\xi w+\partial_\tau w)=[-(\partial_\xi w+\partial_\tau w)](2w\partial_{\eta }^2w) \geq  -24\delta^2 w\quad at \quad z_{max},
$$ where $\delta^2$  is small enough for $L_0 (\partial_\xi w+\partial_\tau w)$ to be controlled. We deduce $g\leq0$ so that $\partial_\xi w+\partial_\tau w\leq \frac{\delta}{2}w$. This is
a stronger estimate than \eqref{eta2nB} so our bootstrap argument is complete.

In the main body of the paper, we  construct a series of estimates by the maximum principle where we use the notation
 \begin{align}\label{cc0729730}
 C_0=c_0^{-1}
\end{align}
for clarity of the dependence.

 \subsection{Existing Literature}

For steady flows, Oleinik\cite{Olei} proved the basic theorem regarding existence and uniqueness  of strong solutions by the maximum principle. In particular, a global-in-$x$ solution exists in the case of favorable pressure gradient. For higher regularity, \cite{GI1} established higher regularity through energy method and then \cite{YWZ} established global $C^\infty$ regularity by the maximum principle method. For stability, as mentioned in \eqref{Serrinconvergence},
\cite{Serrin} proved the convergence to Blasius solution and \cite{I1} established refined asymptotics
for
 perturbation of the Blasius profile.
 In
the case of adverse pressure gradient, \cite{DM} as well as \cite{SWZ} justified the physical phenomenon of boundary layer separation. For validity, \cite{GI1}-\cite{GI2},\cite{GM18} first  validated Prandtl layer expansions,  with the no-slip condition for small $x$, while the main concern of \cite{GI1}-\cite{GI2} is the Blasius-like solutions and  \cite{GM18}'s main concern are shear flows.
\cite{GZ20} generalized  \cite{GI1}-\cite{GI2} to the case of non-shear Euler flows  for small $x$.
Recently, \cite{IN3} validated Prandtl's boundary layer theory \textit{globally}-in-$x$ for a large class of steady state solutions including the
 Blasius profile.
The readers can  also see \cite{GN1},\cite{I2,I3,I4,I5,I6} for validity with the assumption of a moving boundary.

For unsteady flows, we only provide an incomplete list of related works among the huge existing literature. For the wellposedness, under the monotonicity condition $\partial_y u_0>0,  \partial_y u_1>0$, \cite{Olei} proved global-in-$t$ existence of solutions on $[0,X]$ for small $X$ and local-in-$t$ existence for any $X\in\R_+$. \cite{XZ,XZZ} proved global-in-$t$ existence  of smooth solutions on $[0,X]$ for any positive constant $X$. These works employed Crocco transformation, while \cite{AWXY} and \cite{MW} recovered local-in-$t$
existence by the Nash-Moser iteration scheme and a new nonlinear energy estimate respectively without using Crocco transformation. Without monotonicity assumption,  the wellposedness  was  established in  analyticity setting or Gevrey setting and readers can see \cite{DG}, \cite{GM}, \cite{IV}, \cite{LY}, \cite{LCS},  etc. In Sobolev spaces, Prandtl equations are generally illposed without monotonicity assumption: \cite{GD} and \cite{GerN}. For finite time blowup results and boundary layer separation in the unsteady setting, we point readers towards \cite{CGIM}, \cite{CGM}, \cite{EE},  \cite{HH}, \cite{KVW}, \cite{WZ}. For validity of the expansions, the stability of the expansions was established  in the
 analyticity setting or the  Gevrey setting \cite{GMM}, \cite{SC1}, \cite{SC2}. The reader could also see \cite{Ka84}, \cite{Ma},\cite{WWZ17}, etc. For the Sobolev setting, invalidity  of expansions  is established in \cite{GGN1}, \cite{GGN2}, \cite{GGN3}, \cite{GreN}, \cite{GreN1}, \cite{GN},  etc.

\section{Preliminaries}\label{1}

\begin{proposition} \label{prop:monoB}
Let $w$ be a solution of \eqref{eq:Pran-Cro}  with
\begin{align}
 w_1\leq C_0(1-\eta)\sqrt{-\ln{(\mu(1-\eta))}},\quad w_0\leq C_0(1-\eta)\sqrt{-\ln{(\mu(1-\eta))}}
\end{align} for some positive constants $C_0$ and $\mu,$ $0<\mu<\frac{1}{100}.$
Then
\begin{align}\label{wupB}
 w\leq C_0(1-\eta)\sqrt{-\ln{(\mu(1-\eta))}} \quad in\quad D.
\end{align}
\end{proposition}\begin{proof}Set
  \begin{align}\label{g4-1}
 g=C_0(1-\eta)\sqrt{-\ln{(\mu(1-\eta))}}-w+\varepsilon\tau\quad in \quad D,
\end{align}
where $\varepsilon$ is any sufficiently small positive constant.
Our goal is to prove that $g$ is nonnegative in $D$ by the maximum principle.

\textit{Step 1} We will prove the minimum of $ g$  can only be attained on $\overline{D}\setminus D.$

If $g$ attains its minimum at a point $z_{min}\in D,$ then
$\partial_\tau g\leq 0,\,\,\partial_\xi g\leq 0,\,\,\partial_{\eta}^2 g\geq 0 $ at $z_{min}$ and thus
 \begin{align}\label{l0g4}
 L_0 g(z_{min})\geq 0,
\end{align}
where we recall $L_0=-\partial_\tau-\eta\partial_\xi+w^2\partial_{\eta}^2$ in \eqref{defl0}.  However, by \eqref{g4-1},
  \begin{align*}
  L_0 g&=-w^2C_0(\frac{1}{2(1-\eta)\sqrt{-\ln{(\mu(1-\eta))}}}
  +\frac{1}{4(1-\eta)(\sqrt{-\ln{(\mu(1-\eta))}})^3}) -\varepsilon\\
  &\leq -\varepsilon<0 \quad
in\quad D,\end{align*} since $L_0 w=0$ in $D,$ which contradicts to \eqref{l0g4}.
  Hence, the minimum can only be attained on $\overline{D}\setminus D.$

\textit{Step 2} We will prove $g$ cannot attain its negative minimum on $\eta=0.$

If $g$ attains its negative minimum at a point $z_{min}\in \{\eta=0\},$ then there are two cases:

(1) If $w(z_{min})=0$, then $g(z_{min})\geq0 $ by \eqref{g4-1}, which leads to a contradiction.

(2) If $w(z_{min})\neq 0$, then
 \begin{align}\label{etawzmin}
 \partial_\eta w(z_{min})=0,
\end{align}
by the boundary condition $w\partial_\eta w|_{\eta=0}=0$ in \eqref{eq:Pran-Cro}. Then, by \eqref{etawzmin} and  \eqref{g4-1},
 \begin{align*}
 \partial_\eta g(z_{min})=-C_0(\sqrt{-\ln \mu}-\frac{1}{2\sqrt{-\ln \mu}})<0
\end{align*}
and therefore $z_{min}$ is not a minimum point of $g,$ which is a contradiction.
Hence, $g$ cannot attain its negative minimum on $\eta=0.$

In summary,  a negative minimum of $g$ can only be attained on  the boundary $\{\tau=0\}\cup\{\xi=0\}\cup\{\eta=1\}$.
Since $g|_{\eta=1}\geq0$, $g|_{\tau=0}\geq0$ and $g|_{\xi=0}\geq0$ by \eqref{g4-1}, \eqref{eq:Pran-Cro} and the initial and boundary data, we have  $g\geq0$ in $\overline{D}$. Letting $\varepsilon\to 0$,  we get $w\le C_0(1-\eta)\sqrt{-\ln{(\mu(1-\eta))}}$ in $\overline{D}$.

\end{proof}
\begin{lemma} \label{prop:mono1B}
Let $w$ be a solution of \eqref{eq:Pran-Cro} satisfying \eqref{ulbwB}.
Then there exists a positive constant $c_T$ depending on $c_0,$ $C_0,$ $\mu$ and $T$ such that
\begin{align*}
 w(\tau,\xi, \eta)\geq c_T(1-\eta)\quad in \quad \overline{D}.
\end{align*}
In particular, $w>0$ in $[0,T]\times[0,X]\times[0,1)$ and
\begin{align}\label{detaw=0}
    \partial_\eta w=0 \quad on \quad \eta=0.
\end{align}
\end{lemma}\begin{proof}
Set
\begin{align}\label{g4-2}
g=-e^{-K \tau} \al e^{\eta }(1-\eta)+\varepsilon(1-\eta)+\varepsilon\tau+w\quad in \quad \bar{D},
\end{align}
where
 $\al\in(0,\frac{c_0}{e}]$ and $ K$  are positive constants to be determined and $\varepsilon$ is any  small positive constant.
Our goal is to prove that $g$ is nonnegative in $D$ by the maximum principle.

 \textit{Step 1 } We will prove
\begin{align}\label{claim}
 w>0\quad on \quad [0,T]\times[0,X]\times\{\eta=0\}.
\end{align}

In fact, if \eqref{claim} is false, then there exists a constant $T^*\in(0,T]$ such that
\begin{align}\label{T*5}
    T^*=\sup\{s\in[0,T]|w>0\quad for \quad (\tau,\xi,\eta)\in[0,s]\times[0,X]\times \{\eta=0\}\}
\end{align} and $w( T^*,\xi,0)=0$ for some $\xi\in[0,X].$
Here we note $T^*>0$, since $w_0>0$ on $[0,X]\times \{\eta=0\}.$
Set
 \begin{align}\label{ddmu}
 D_{T^*-\mu}=(0,T^*-\mu]\times(0,X]\times (0,1),
\end{align}
where $\mu$ is any sufficiently small positive constant.

\textit{Step 1.1} We will prove $g$ can not attain a negative minimum
on $\overline{D_{T^*-\mu}}\setminus D_{T^*-\mu}.$

By the definition of $T^*,$ $w>0$ on $\overline{D_{T^*-\mu}}\cap\{\eta=0\}$ and thus
$\partial_\eta w=0$ on $ \overline{D_{T^*-\mu}}\cap\{\eta=0\}$ by $w\partial_\eta w|_{\eta=0}=0$ in \eqref{eq:Pran-Cro}.
Hence, by \eqref{g4-2}, $\partial_\eta g=-\varepsilon<0$  on $ \overline{D_{T^*-\mu}}\cap\{\eta=0\}$ and thus $g$ does not attain its minimum on  $ \overline{D_{T^*-\mu}}\cap\{\eta=0\}$.

Moreover, by $\al\in(0,\frac{c_0}{e}]$, \eqref{g4-2}, \eqref{eq:Pran-Cro} and \eqref{ulbwB}, we have
\begin{align}\label{ibg}
g \geq0\quad on\quad (\{\tau=0\}\cup\{\xi=0\}\cup\{\eta=1\})\cap\overline{D}.
\end{align}

Therefore, $g$ can not attain a negative minimum
on $\overline{D_{T^*-\mu}}\setminus D_{T^*-\mu}.$

\textit{Step 1.2} We will prove  no minimum of $g$ in $\overline{D_{T^*-\mu}}$ is  attained in $D_{T^*-\mu}.$

In fact, taking $K$ large  depending only on $C_0$ and $\mu$, we have, by  \eqref{g4-2},
\begin{align}\label{dmuB}\begin{split}
    L_0g&=-\varepsilon+[w^2(1+\eta)-K(1-\eta)]e^\eta e^{-K \tau}\al
    \\&\leq \Big[-K(1-\eta)+2C_0^2(1-\eta)^2(-\ln{(\mu(1-\eta))})\Big]e^\eta e^{-K \tau}\al-\varepsilon\\&<0 \quad in \quad D,\end{split}
\end{align} since $L_0 w=0$ in $D,$ where we recall $L_0$ in \eqref{defl0}.
Hence, no minimum of $g$ is  attained in $D_{T^*-\mu}.$

In summary,
$g\geq 0$ in $\overline{D_{T^*-\mu}}.$ Letting $\varepsilon\to 0,$ we have
\begin{align*}
w\geq e^{-K \tau} \al e^{\eta }(1-\eta)\geq e^{-K T^*}\al(1-\eta)\quad in \quad \overline{D}.
\end{align*} In particular, $w( T^*,\xi,0)>0$ which contradicts to the definition of $T^*$ in \eqref{T*5}. Hence, \eqref{claim} holds and we complete step 1.

 \textit{Step 2 } We will prove
 $g $ is nonnegative in  $\overline{D}$ by applying \eqref{claim}.

\textit{ Step 2.1} By \eqref{dmuB},
 no minimum of $g$ is  attained in $D.$

\textit{ Step 2.2}  We will prove $g$ can not attain a negative minimum
on $\overline{D}\setminus D$. Then, combining it with the result obtained in step 1, we have $g$ is nonnegative in  $\overline{D}$.

By \eqref{claim},
 \begin{align}\label{redeta0}
 \partial_\eta w=0 \quad on\quad \overline{D}\cap \{\eta=0\}
\end{align}
by $w\partial_\eta w|_{\eta=0}=0$ in \eqref{eq:Pran-Cro} and therefore by \eqref{g4-2} $\partial_\eta g=-\varepsilon<0$  on $\bar{D}\cap\{\eta=0\}.$ Therefore, $g$ does not attain a minimum on  $ \bar{D}\cap\{\eta=0\}$. Hence, by \eqref{ibg},
$g\geq 0$ in
 $\bar{D}.$

  Letting $\varepsilon\to 0,$ we have
\begin{align*}
w\geq e^{-K \tau} \al e^{\eta }(1-\eta)\geq e^{-K T}(1-\eta)\al \quad in \quad \overline{D},
\end{align*}
and we have \eqref{detaw=0} by \eqref{redeta0}.

\end{proof}

\begin{proposition}\label{prop27B}
Let $w$ be a solution of \eqref{eq:Pran-Cro} in $D$  satisfying the conditions in Theorem \ref{w1tB}. Then for any  constant $\alpha\in(0,\alpha_0),$ we have
\begin{align}\label{al10B}
   \displaystyle\lim_{\eta\to 1}\frac{ \partial_\xi w}{(1-\eta)^\alpha}=0,\quad \displaystyle\lim_{\eta\to 1}\frac{ \partial_\tau w}{(1-\eta)^\alpha}=0 \quad for \quad (\tau,\xi)\in[0,T]\times[0,X].
\end{align}
\end{proposition}
\begin{proof}
Our goal is to prove $-e^{K  T}(b\delta+C_1)(1-\eta)^{\alpha_0}\leq
\partial_{\tau,\xi} w \leq e^{K  T}(b\delta+C_1)(1-\eta)^{\alpha_0}
$ in $D,$  which implies \eqref{al10B}.\textit{ Here   $\partial_{\tau,\xi} w$ stands for $\partial_\xi w$
 or  $\partial_\tau w.$}

 Since the proofs for the two inequalities are similar, we only prove one direction
\begin{align}\label{leq6}
\partial_{\tau,\xi} w \leq e^{K  \tau}(b\delta+C_1)(1-\eta)^{\alpha_0}\quad in \quad D.
\end{align}

\textit{Step 1} We will give the equation of $\partial_{\tau,\xi} w$  and then construct a new function $g$ such that to prove \eqref{leq6} is equivalent to prove $g\leq 0$ in $D.$

By \eqref{eq:Pran-Cro}, \eqref{detaw=0} and the  assumption in Theorem \ref{w1tB}, we have
 \begin{equation}\label{fieqB1}
\left\{\begin{aligned}
&-\partial_\tau  \partial_{\tau,\xi} w-\eta \partial_\xi \partial_{\tau,\xi} w+w^2\partial_{\eta }^2 \partial_{\tau,\xi} w+(2w\partial_{\eta }^2w)\partial_{\tau,\xi} w=0\quad (\tau,\xi,\eta)\in D,\\
&\partial_\eta \partial_{\tau,\xi} w\mid _{\eta=0}=0,\quad \displaystyle\lim_{\eta\to 1} \partial_{\tau,\xi} w= 0 ,
\end{aligned}\right.
\end{equation}where  $\partial_{\tau,\xi} w$ stands for $\partial_\xi w$
 or  $\partial_\tau w.$ By \eqref{epsilon1B},\begin{align}\label{2.156}
   |\partial_{\tau,\xi} w|\leq (b\delta+C_1) (1-\eta)^{\alpha_0}\quad on\quad \tau=0 \quad and \quad \xi=0,
\end{align}

Next, by Lemma \ref{prop:mono1B}, we have
  $w>0$ in $[0,T]\times[0,X]\times[0,1)$. Then by \eqref{ww-10uB}, we have $
  w\partial_\eta^2 w$ is  continuous
in $\overline{D}$, which implies there exists a positive constant $A_{T,X}$ depending on $X$ and $T$ such that\begin{align}\label{AXB}
| w\partial_\eta^2 w|\leq A_{T,X}.
\end{align}
Set
\begin{align}\label{g5}
g = \partial_{\tau,\xi} w e^{-K  \tau}-(b\delta+C_1)(1-\eta)^{\alpha_0},
\end{align}
 where
 \begin{align}\label{betaiB}
 K =2A_{T,X}+1.
\end{align}

\textit{Step 2} We will  prove  $g\leq 0$  in $D$ by the maximum principle.

\textit{Step 2.1} We will prove $g $ cannot attain a positive maximum on $\overline{D}\setminus D.$

By \eqref{fieqB1} and \eqref{g5}, $\partial_\eta g \mid _{\eta=0}=(b\delta+C_1)\alpha_0>0,$ $g $ cannot have a maximum on $\eta=0$. By \eqref{eta1xi}, \eqref{2.156} and \eqref{g5}, $g \leq 0$ on $ \{\tau=0\}\times[0,X]\times[0,1]\cup[0,T]\times\{\xi=0\}\times[0,1]
\cup[0,T]\times[0,X]\times\{\eta=1\}.$
Hence, $g $ cannot attain a positive maximum on $\overline{D}\setminus D.$

\textit{Step 2.2} We will prove $g $ cannot attain a positive maximum in $ D.$

 If $g $ attains its positive maximum at a point $z_{max}\in D,$
 then
  \begin{align}\label{zmaxeta0B}
\partial_{\tau,\xi} w>0,\,\,\partial_\tau g \geq 0, \,\, \partial_\xi g \geq 0,\,\, w^2\partial_{\eta }^2 g \leq 0\quad at \quad z_{max},
\end{align}
and  thus
\begin{align}\label{l0g5-2}
 L_0g (z_{max})\leq 0,
\end{align}
where we recall $L_0  $ in \eqref{defl0}.
However, by \eqref{AXB}, \eqref{fieqB1}, \eqref{g5}, \eqref{betaiB}  and  $\partial_{\tau,\xi} w(z_{max})>0$ in \eqref{zmaxeta0B}, $$L_0 g  =(-2w\partial_{\eta }^2w+K )\partial_{\tau,\xi} we^{-K  \tau}+w^2(b\delta+C_1)(1-\alpha_0)\alpha_0(1-\eta)^{\alpha_0-2}>0\quad at \quad z_{max},$$ which  contradicts to \eqref{l0g5-2}.

In summary, $g $ does not have a positive maximum in $\overline{D}$ and  therefore $g\leq 0$  in $D$. Then we complete step 2.

Hence, $\partial_{\tau,\xi} w\leq (b\delta+C_1) (1-\eta)^{\alpha_0}e^{K  T}$ in $\overline{D}$.

\end{proof}

\section{ Invariant set \eqref{keying1B} }\label{2}
In this section, we will prove Lemma \ref{unilbd1} and Proposition \ref{eta2negVIPB}. Theorem \ref{w1tB} is a direct consequence of them.
\subsection{Proof of Proposition \ref{eta2negVIPB}}

Before proving Proposition \ref{eta2negVIPB}, we prove two tool lemmas.

\begin{lemma}\label{yuanstep1}
Let $w$ be the Oleinik monotone solution of \eqref{eq:Pran-Cro} in $D$  satisfying the conditions in Theorem \ref{w1tB}. Assume, for some positive  constant $T^*\in(0,T],$
\begin{align}\label{wb1-etaB}
 w\geq b(1-\eta),\quad w\partial_\eta^2 w\leq 2\delta\quad in \quad [0,T^*]\times[0,X]\times[0,1].
\end{align}

Then \begin{align}\label{nwato1-1}\begin{split}
&-C_1(1-\eta)^{\alpha_0}\leq\partial_\xi w+\partial_\tau w\leq \frac{(\delta b)^{\frac{1}{\alpha_0}}}{K}(1-\eta)^{\alpha_0}\quad in \quad [0,T^*]\times[0,X]\times[0,1] ,\\& -C_1(1-\eta)^{\alpha_0}\leq\partial_\tau w\leq b\delta(1-\eta)^{\alpha_0} \quad in \quad  [0,T^*]\times[0,X]\times[0,1].\end{split}
\end{align}

\end{lemma}
\begin{proof}
Recall \begin{equation}\label{fieqB}
\left\{\begin{aligned}
&L \partial_{\tau,\xi} w=0\quad (\tau,\xi,\eta)\in D,\\
&\partial_\eta \partial_{\tau,\xi} w\mid _{\eta=0}=0,\quad \displaystyle\lim_{\eta\to 1} \partial_{\tau,\xi} w= 0 ,
\end{aligned}\right.
\end{equation} where  $\partial_{\tau,\xi} w$ stands for $\partial_\xi w$
 or  $\partial_\tau w$  and we recall $L$ in \eqref{defl}. Set
 \begin{align}
    D_{T^*}= (0,T^*]\times(0,X]\times(0,1).
 \end{align}

    Since the coefficient of the zero order term in  $L$ ( see \eqref{defl} for $L$'s definition) is $2w\partial_{\eta }^2w$ which   can take positive values, $L$ is not suitable for applying the maximum principle. Hence, instead of constructing  barrier functions  for $ \partial_{\tau,\xi} w$, we introduce a new function with the form $ \frac{\partial_{\tau,\xi} w}{v}$ so that the concavity of $v$ yields good terms to reconcile the bad effect from the positivity of $2w\partial_{\eta }^2w$.  Our proof of \eqref{nwato1-1} is divided into 3 steps.

   \textit{Step 1} We will prove
   \begin{align}\label{step1-611}
      -C_1(1-\eta)^{\alpha_0}\leq\partial_\tau w \quad in\quad D_{T^*}.
   \end{align}

   For any $\alpha\in(\frac{\alpha_0}{2},\alpha_0),$ set
   \begin{align}\label{val}
    v=(1-\eta)^\alpha
   \end{align} and
\begin{align}\label{g6}
 g=\frac{ \partial_\tau w}{v}.
\end{align} To prove \eqref{step1-611}, we will prove   \begin{align}\label{step1-6}
      -C_1\leq  g\quad in\quad D_{T^*}
   \end{align}
by  the maximum principle.

\textit{Step 1.1} We will prove $g$ can only attain a negative minimum on $\overline{D_{T^*}}\setminus D_{T^*}.$

 If $g$ attains its negative minimum at a point $z_{min}\in D_{T^*},$
 then
  \begin{align}\label{zmineta01}
  g<0,\,\,\partial_\eta g=0,\,\,\partial_\tau g\leq 0, \,\, \partial_\xi g\leq 0,\,\, w^2\partial_{\eta }^2 g\geq 0\quad at \quad z_{min},
\end{align}
and  thus
 \begin{align}\label{l0gp1}
 L_0g(z_{min})\geq 0,
\end{align}
where we recall $L_0 $ in \eqref{defl0}.

However, \textit{since}, by  \eqref{fieqB} and \eqref{g6}, \begin{align}\label{ldtauw}
   0= L( \partial_\tau w)=vL_0 g + g Lv+2 w^2\partial_{\eta }v\partial_{\eta }g\quad in \quad D_{T^*},
 \end{align} where we recall $L\cdot=(2w\partial_{\eta }^2w)\cdot-\partial_\tau \cdot -\eta \partial_\xi\cdot +w^2\partial_{\eta }^2 \cdot  $ in \eqref{defl}  and $L_0 \cdot=-\partial_\tau \cdot -\eta \partial_\xi\cdot +w^2\partial_{\eta }^2 \cdot $ in \eqref{defl0} and by \eqref{val},
 \begin{align}\label{lvneg}\begin{split}
 \frac{Lv}{v}\leq&\frac{1}{v}(w^2\alpha(\alpha-1)(1-\eta)^{\alpha-2}
 +4\delta(1-\eta)^{\alpha})\\ \leq&\frac{1}{v}(-b^2(1-\eta)^2\alpha(1-\alpha)(1-\eta)^{\alpha-2}
 +4\delta(1-\eta)^{\alpha})\\<&0\quad in \quad D_{T^*},\end{split}
 \end{align} where we used the assumption \eqref{wb1-etaB} and $\delta\leq \frac{1}{10}\min_{\alpha\in[\frac{\alpha_0}{2},\alpha_0]}\alpha(1-\alpha)b^2$ such that $$ w^2\geq b^2(1-\eta)^2,\quad 2w\partial_\eta^2 w\leq 4\delta<\alpha(1-\alpha)b^2 \quad in \quad D_{T^*},$$
 \textit{we have}, by \eqref{zmineta01} and $g(z_{min})<0$,
 \begin{align*}
  L_0 g=  -(g \frac{Lv}{v}+2 \frac{w^2}{v}\partial_{\eta }v\partial_{\eta }g)=-g \frac{Lv}{v}<0\quad at \quad z_{min} ,
 \end{align*} which contradicts to \eqref{l0gp1}.
 Hence, $g$ does not have a negative minimum in $D_{T^*}.$

\textit{Step 1.2} We will prove $g$ does not have a negative minimum on $\eta=0.$

Since, by \eqref{val}, \begin{align}\label{veta0}
    \partial_\eta v=-\alpha\quad on \quad \eta=0,
 \end{align} and thus by \eqref{fieqB} and \eqref{g6},
 \begin{align}\label{etadtauw}
    0=\partial_{\eta } ( \partial_\tau w)=v\partial_{\eta } g+g\partial_{\eta }v
    =v\partial_{\eta } g-\alpha g \quad on \quad \eta=0,
 \end{align}
if $g$ attains its negative minimum at a point $z_{min}\in \{\eta=0\},$ then
 $\partial_{\eta } g(z_{min})=\frac{\alpha g}{v}(z_{min})<0.$
  Hence, $g$ does not have a negative minimum on $\eta=0.$

 In summary, $g$ can only attain its negative minimum  on $\overline{D_{T^*}}\cap(\{\xi=0\}\cup\{\tau=0\}\cup\{\eta=1\}).$
 By \eqref{epsilon1B}, $g\geq -C_1$ on $\xi=0$ and $\tau=0.$  By \eqref{al10B}, $\lim_{\eta\to 1} g= 0.$
 Hence,  $g\geq -C_1$ in $\overline{D_{T^*}}.$
Letting $\alpha$ go to $\alpha_0$, we complete step 1.

\textit{Step 2} We will prove
 \begin{align}\label{s2-6}
 \partial_\tau w\leq  b\delta(1-\eta)^{\alpha_0}\quad in \quad D_{T^*}.
\end{align}

Set  \begin{align}\label{DG7}
  G=g- b\delta
\end{align} where we recall $g$ in \eqref{g6}. To prove \eqref{s2-6}, we will prove   \begin{align}\label{step1-6}
      G\leq0\quad in\quad D_{T^*}
   \end{align}
by the maximum principle.

\textit{Step 2.1} We will prove $G$ cannot  attain a positive maximum in $ D_{T^*}.$

If $G$ has a positive maximum at some point $z_{max}\in D_{T^*},$ then, by \eqref{DG7},   \begin{align}\label{s1gpos}
 g>0,\,\, \partial_\tau G\geq 0, \,\, \partial_\xi G\geq 0,\,\, w^2\partial_{\eta }^2 G\leq 0\quad at \quad z_{max},
\end{align} implying
  \begin{align}\label{l0dG}
   L_0 G(z_{max})\leq 0
\end{align} where we recall $L_0$ in \eqref{defl0}
and
\begin{align}\label{detagmax}
0=\partial_\eta G=\partial_\eta g\quad at \quad z_{max}.
\end{align}

However, by \eqref{lvneg}, \eqref{ldtauw}
   $g(z_{max})>0$ in \eqref{s1gpos}  and $\partial_\eta g(z_{max})=0$ in \eqref{detagmax},\begin{align*}
  L_0 g=  -(g \frac{Lv}{v}+2 \frac{w^2}{v}\partial_{\eta }v\partial_{\eta }g)=-g \frac{Lv}{v}>0\quad at \quad z_{max} .
 \end{align*}  Then $$L_0 G=L_0 g>0\quad at \quad z_{max},$$  which  contradicts to \eqref{l0dG}. Hence, $G$ cannot have a positive maximum in $D_{T^*}.$

\textit{Step 2.2} We will prove $G$ does not have a positive maximum on $\eta=0.$

 If $G$ attains its positive maximum at a point $z_{max}\in \{\eta=0\},$ then $ g(z_{max})>0$ and therefore, by \eqref{etadtauw},
    \begin{align}
   \partial_{\eta } G(z_{max})=\partial_{\eta } g(z_{max})=\frac{\alpha g}{v}(z_{max})>0.
\end{align}
Hence, $G$ cannot have a positive maximum on $\eta=0.$

  In summary, $G$ can only attain its positive maximum  on $\overline{D_{T^*}}\cap(\{\xi=0\}\cup\{\tau=0\}\cup\{\eta=1\}).$
However, by \eqref{epsilon1B}, $G\leq 0$ on $\xi=0$ and $\tau=0.$ By \eqref{g6} and \eqref{al10B}, $\lim_{\eta\to 1} g= 0,$ implying  $G\leq 0$ on $\eta=1.$
Then  $G\leq0$ in $\overline{D_{T^*}}.$  Letting $\alpha$ go to $\alpha_0$, we complete step 2.

\textit{Step 3} We will prove $-C_1(1-\eta)^{\alpha_0}\leq\partial_\xi w+\partial_\tau w\leq \frac{(\delta b)^{\frac{1}{\alpha_0}}}{K}(1-\eta)^{\alpha_0}$  in $D_{T^*}$. Its proof is similar to the proof in step 1-2, so we only prove one direction,
\begin{align}\label{s3-6}
\partial_\xi w+\partial_\tau w\leq \frac{(\delta b)^{\frac{1}{\alpha_0}}}{K}(1-\eta)^{\alpha_0}\quad  in \quad D_{T^*},
\end{align}
for readers' convenience.
For any $\alpha\in(\frac{\alpha_0}{2},\alpha_0),$ set $$v=(1-\eta)^\alpha,$$
\begin{align}\label{gsum}
g_{sum}=\frac{  \partial_\xi w+\partial_\tau w}{v}
\end{align} and
\begin{align}\label{dGsum}
G_{sum}=g_{sum}- \frac{(\delta b)^{\frac{1}{\alpha_0}}}{K}.
\end{align}
To prove \eqref{s3-6}, we will prove   \begin{align}\label{step1-6}
      G_{sum}\leq  0\quad in\quad D_{T^*}
   \end{align}
by the maximum principle.

\textit{Step 3.1} We will prove $ G_{sum}$ can not attain a  positive maximum in $ D_{T^*}.$

If $G_{sum}$ has a positive maximum at some point $z_{max}\in D_{T^*},$ then   \begin{align}\label{gsumpos}
  g_{sum}>0,\,\,\partial_\tau G_{sum}\geq 0, \,\, \partial_\xi G_{sum}\geq 0,\,\, w^2\partial_{\eta }^2 G_{sum}\leq 0\quad at \quad z_{max},
\end{align} implying
  \begin{align}\label{l0dGsum}
  L_0 G_{sum}(z_{max})\leq 0
\end{align}
 where we recall $L_0$ in \eqref{defl0} and
\begin{align}\label{detadagsum}
0=\partial_\eta G_{sum}=\partial_\eta g_{sum}\quad at \quad z_{max}.
\end{align}

However, since, by \eqref{gsum} and \eqref{fieqB}, \begin{align*}
   0= L( \partial_\xi w+\partial_\tau w)=vL_0 g_{sum} + g_{sum} Lv+2 w^2\partial_{\eta }v\partial_{\eta }g_{sum}\quad in \quad D_{T^*},
 \end{align*} where  we recall $L $ in \eqref{defl} and $L_0 $ in \eqref{defl0}, we have, by \eqref{lvneg}, $g_{sum}(z_{max})>0$ in \eqref{gsumpos} and \eqref{detadagsum},
 \begin{align*}
 L_0 G_{sum}= L_0 g_{sum}=  -(g_{sum} \frac{Lv}{v}+2 \frac{w^2}{v}\partial_{\eta }v\partial_{\eta }g_{sum})=-g_{sum} \frac{Lv}{v} >0\quad  at \quad z_{max},
 \end{align*}
  which contradicts to \eqref{l0dGsum}. Hence, $G$ cannot have a positive maximum in $D_{T^*}.$

  \textit{Step 3.2} We will prove $G$ does not have a  positive maximum on $\eta=0.$

 Since, by \eqref{val}, \eqref{fieqB} and \eqref{gsum},
  \begin{align*}
    0=\partial_{\eta } (\partial_\xi w+\partial_\tau w)=v\partial_{\eta } g_{sum}+g_{sum}\partial_{\eta }v
    =v\partial_{\eta } g_{sum}-\alpha g_{sum} \quad on \quad \eta=0,
 \end{align*}
 if $G_{sum}$ attains its positive maximum at a point $z_{max}\in \{\eta=0\},$ then $ g_{sum}(z_{max})>0$ and therefore
 $$\partial_{\eta } G_{sum}(z_{max})=\partial_{\eta } g_{sum}(z_{max})=\frac{\alpha g_{sum}}{v}(z_{max})>0.$$   Hence, $G_{sum}$ cannot have a positive maximum on $\eta=0.$

  In summary, $G_{sum}$ can only attain its positive maximum on $\overline{D_{T^*}}\cap(\{\xi=0\}\cup\{\tau=0\}\cup\{\eta=1\}).$
  By \eqref{epsilon1B}, $G_{sum}\leq 0$ on $\xi=0$ and $\tau=0.$ By \eqref{al10B}, $\lim_{\eta\to 1} g_{sum}= 0.$ Hence,  $G_{sum}\leq 0$ on $\eta=1.$
 Hence,  $G_{sum}\leq0$ in $\overline{D_{T^*}}.$  Letting $\alpha$ go to $\alpha_0$,  we have proved $\partial_\xi w+\partial_\tau w\leq \frac{(\delta b)^{\frac{1}{\alpha_0}}}{K}(1-\eta)^{\alpha_0}$ in $D_{T^*}.$

\end{proof}
\begin{lemma}\label{0618}
Let $w$ be the Oleinik monotone solution   to \eqref{eq:Pran-Cro} in $D$ satisfying  \eqref{ulbwB}-\eqref{epsilon1B}.  Assume, for some positive  constant $T_*\in(0,T],$
\begin{align}\label{nwato1}\begin{split}
&-C_1(1-\eta)^{\alpha_0}\leq\partial_\xi w+\partial_\tau w\leq \frac{(\delta b)^{\frac{1}{\alpha_0}}}{K}(1-\eta)^{\alpha_0},\\& -C_1(1-\eta)^{\alpha_0}\leq\partial_\tau w\leq b\delta(1-\eta)^{\alpha_0},
 \\&\,\,w\geq b(1-\eta),\quad w\partial_\eta^2 w\leq 2\delta\quad in \quad D_{T_*},\end{split}
\end{align}  where
 \begin{align}\label{DT*}
   D_{T_*}=(0,T_*]\times(0,X]\times(0,1).
\end{align}
Then $$\partial_\xi w+\partial_\tau w\leq \frac{\delta}{2}e^{\xi-X}w \quad in \quad D_{T_*}.$$
\end{lemma}
\begin{proof}
First, by \eqref{nwato1}, we have
\begin{align}\label{onpxi}
\partial_\xi w+\partial_\tau w\leq \frac{(\delta b)^{\frac{1}{\alpha_0}}}{2C_1^{\frac{1}{\alpha_0}-1}e^X}(1-\eta)^{\alpha_0}\leq e^{-X}\frac{b\delta}{2}(1-\eta)\leq e^{-X}\frac{\delta}{2}w \quad in \quad \overline{D^*_0},
\end{align} where $D^*_0=\{\eta|(1-\eta)^{\alpha_0}\geq\frac{b\delta}{C_1}\}\cap D_{T_*}$ where we recall  $D_{T_*}$ in \eqref{DT*}.

Then we only need to consider in the domain $\{\eta|(1-\eta)^{\alpha_0}\leq\frac{b\delta}{C_1}\}\cap \overline{D_{T_*}}.$ Set
\begin{align}\label{D1*}
D^*_1:=\{\eta|(1-\eta)^{\alpha_0}<\frac{b\delta}{C_1}\}\cap D_{T_*},
\end{align}
where we recall  $D_{T_*}$ in \eqref{DT*} and
\begin{align}\label{g9}
g=(\partial_\xi w+\partial_\tau w)-e^{-X}\frac{\delta}{2}w e^{\xi}.
\end{align}
Then our goal is to prove $g\leq 0$ in $D^*_1$ where we recall $D^*_1$ in
\eqref{D1*}.

\textit{Step 1} We will prove $g\leq 0$ on $\overline{D^*_1}\setminus D^*_1.$

By \eqref{onpxi} and \eqref{nwato1}, $$g\leq 0\quad on \quad\{\eta|(1-\eta)^{\alpha_0}=\frac{b\delta}{C_1}\}\cap  \overline{D_{T_*}}\quad and \quad\{\eta=1\}\cap  \overline{D_{T_*}}.$$
Moreover, by \eqref{epsilon1B} and $\delta b<1,$
\begin{align*}
\partial_\xi w+\partial_\tau w\leq \frac{(\delta b)^{\frac{1}{\alpha_0}}}{2e^X}(1-\eta)\leq e^{-X}\frac{b\delta}{2}(1-\eta)\leq e^{-X}\frac{\delta}{2}w \quad on \quad \tau=0\quad and \quad \xi=0.
\end{align*} Then $$g\leq0\quad on \quad(\{\tau=0\}\cup\{\xi=0\})\cap \overline{D_{T_*}}.$$ Hence we have $g\leq 0$ on $\overline{D^*_1}\setminus D^*_1$ by the definition of $D^*_1$ in \eqref{D1*}.

\textit{Step 2} We will prove $g$ cannot attain a positive maximum in $D^*_1.$

If $g$ attains its positive maximum at some point $z_{max}\in D^*_1,$ then $(\partial_\xi w+\partial_\tau w)(z_{max})>0.$ and
   \begin{align}\label{dltal0g}
  L_0g(z_{max})\leq 0
\end{align}
where we recall $L_0= -\partial_\tau  -\eta \partial_\xi +w^2\partial_{\eta }^2$ in \eqref{defl0}.

Moreover, we claim \begin{align}\label{adv+}
0<(\partial_\xi w+\partial_\tau w)(z_{max})\leq 6\delta w(z_{max}).
\end{align}

\textit{Proof of claim \eqref{adv+}.}  In fact, since
\begin{align*}
 2\delta\geq w\partial_\eta^2 w=& w^{-1}(\eta\partial_\xi w+\partial_\tau w)
 \\ =& w^{-1}[\eta(\partial_\xi w+\partial_\tau w)+(1-\eta)\partial_\tau w] \quad in \quad D_{T_*},
\end{align*} and $-\partial_\tau w\leq C_1(1-\eta)^{\alpha_0}$ by  \eqref{nwato1}, we have
\begin{align*}
\eta(\partial_\xi w+\partial_\tau w)\leq& 2\delta w
-\partial_\tau w(1-\eta)\\
 \leq &2\delta w+C_1(1-\eta)^{\alpha_0+1}\quad in \quad D_{T_*}.
\end{align*}
Then
\begin{align*}
\eta(\partial_\xi w+\partial_\tau w)
 \leq &2\delta w+b(1-\eta)\frac{C_1}{b}(1-\eta)^{\alpha_0}
 \\ \leq &2\delta w+w\frac{C_1}{b}\frac{b\delta}{C_1}
 \\ \leq &3\delta w\quad in \quad D^*_1.\end{align*}
Then $ \partial_\xi w+\partial_\tau w\leq 6\delta w$ in $D^*_1$,
since $(1-\eta)^{\alpha_0}\leq\frac{b\delta}{C_1}\leq(\frac{1}{2})^{\alpha_0}$  by the definition of $ D^*_1$ in \eqref{D1*} and thus
\begin{align}\label{etalbdD1}
\eta\geq \frac{1}{2}\quad in \quad  D^*_1.
\end{align}
\textit{Therefore, we complete the proof of claim \eqref{adv+}.}

Then by \eqref{fieqB} and \eqref{adv+},
\begin{align}\label{xi+tau}\begin{split}
   L_0 (\partial_\xi w+\partial_\tau w)=&[-(\partial_\xi w+\partial_\tau w)](2w\partial_{\eta }^2w) \\
   \geq &4\delta[-(\partial_\xi w+\partial_\tau w)]
   \\ \geq & -24\delta^2 w\quad at \quad z_{max},
\end{split}\end{align}
where we note $-(\partial_\xi w+\partial_\tau w)(z_{max})<0$ by \eqref{adv+}.  By \eqref{defl0},  \eqref{g9} and  $\eta\geq \frac{1}{2}$ in $  D^*_1$ in \eqref{etalbdD1}, we have
\begin{align*}
    L_0 g \geq & -24\delta^2 w+\eta e^{-X}\frac{\delta}{2}w e^{\xi}
   \\ \geq& -24\delta^2 w+\frac{1}{2} e^{-X}\frac{\delta}{2}w e^{\xi}\\
   >& 0
   \quad at \quad z_{max},
\end{align*}
since $L_0 w=0$, $\delta<\frac{1}{96 e^{X}}$ and $w>0$ for $\eta\in(0,1),$ which  contradicts to \eqref{dltal0g}.
Hence, there is no positive maximum in $D^*_1.$

In summary, $g\leq 0$ in $\overline{D^*_1}$ and thus
 $\partial_\xi w+\partial_\tau w\leq e^{-X}\frac{\delta}{2}w e^{\xi}$ in $\overline{D^*_1}$. Combining with \eqref{onpxi}, we have the desired results.

\end{proof}

\begin{proof}[Proof of Proposition \ref{eta2negVIPB}]
 The proof is divided into two steps.

\textit{ Step 1} We will prove \eqref{eta2nB} holds.

 Since, by \eqref{epsilon1B},
it holds
\begin{align*}
 w\partial_\eta^2 w\leq& w^{-1}(\eta\partial_\xi w+\partial_\tau w)
 \\ \leq& w^{-1}[\eta(\partial_\xi w+\partial_\tau w)+(1-\eta)\partial_\tau w]\\ \leq& \frac{b\delta(1-\eta)^{\alpha_0+1}}{c_0(1-\eta)}+w^{-1}\eta\delta b(1-\eta)
 \\ \leq&\frac{5\delta}{4},\quad on \,\,\{\xi=0\} \cup \{\tau=0\},
\end{align*}
if \eqref{eta2nB} is false, then
 there exists a point  $p_0\in(0,T_1]\times(0,X]\times[0,1]$ such that  $w\partial_\eta^2 w(p_0)> 2\delta.$ We define
 \begin{align}\label{T*142}
 T^*=\sup\{s\in[0,T_1]|w\partial_\eta^2 w< 2\delta\quad in \quad [0,s]\times[0,X]\times[0,1]\}.
 \end{align}
  Since $w\partial_\eta^2 w(p_0)> 2\delta$ and $w\partial_\eta^2 w\leq \frac{5}{4}\delta$ on $\tau=0$,
$T^*\in(0,T_1)$. In particular, there is a point $p_1=(T^*,\xi_1,\eta_1)\in\{\tau=T^*\}\times(0,X]\times[0,1]$ such that
$w\partial_\eta^2 w(p_1)= 2\delta.$
Then by  \eqref{wb1-etaB-1}, \eqref{T*142} and Lemma \ref{yuanstep1}, we have\begin{align}\label{nwato1-1-1}\begin{split}
&-C_1(1-\eta)^{\alpha_0}\leq\partial_\xi w+\partial_\tau w\leq \frac{(\delta b)^{\frac{1}{\alpha_0}}}{K}(1-\eta)^{\alpha_0},\\& -C_1(1-\eta)^{\alpha_0}\leq\partial_\tau w\leq b\delta(1-\eta)^{\alpha_0} \quad in \quad  [0,T^*]\times[0,X]\times[0,1].\end{split}
\end{align} Then by \eqref{wb1-etaB-1}, \eqref{T*142}, \eqref{nwato1-1-1} and Lemma \ref{0618}, we have
$$\partial_\xi w+\partial_\tau w\leq \frac{\delta}{2}w \quad in \quad [0,T^*]\times[0,X]\times[0,1].$$

Note $ w^2\partial_\eta^2 w=\eta\partial_\xi w+\partial_\tau w.$
Then we have, by \eqref{wb1-etaB-1}, \begin{align*}
 w\partial_\eta^2 w\leq& w^{-1}(\eta\partial_\xi w+\partial_\tau w)
 \\ \leq& w^{-1}[\eta(\partial_\xi w+\partial_\tau w)+(1-\eta)\partial_\tau w]\\ \leq& \frac{b\delta(1-\eta)^{\alpha_0+1}}{b(1-\eta)}+\frac{\delta}{2}\eta
 \\ \leq&\frac{3}{2}\delta\quad in \quad [0,T^*]\times[0,X]\times[0,1]
\end{align*} and thus
$ w\partial_\eta^2 w\leq \frac{3}{2}\delta $  in  $ [0,T^*]\times[0,X]\times[0,1],$ which contradicts to the definition of $T^*$ in \eqref{T*142}. Therefore, \eqref{eta2nB} holds.

\textit{ Step 2} We will prove \eqref{gn1B} holds.

Since, by \eqref{eta2nB}, \begin{align}
    w\partial_\eta^2 w\leq 2\delta\quad in \quad [0,T_1]\times[0,X]\times[0,1],
\end{align} we have \eqref{gn1B} by  \eqref{wb1-etaB-1}, \eqref{eta2nB}, Lemma \ref{yuanstep1} and Lemma \ref{0618}.
\end{proof}

\subsection{Bootstrap proof of Invariant set}

\begin{lemma}\label{eta0-1}
Let $w$ be a solution of \eqref{eq:Pran-Cro} in $D=(0,T]\times(0,X]\times(0,1)$ with
\begin{align}\label{w0w1a}
w_1\geq a^*(1-\eta),\quad w_0\geq a^*(1-\eta).
\end{align}
 If $w|_{\eta=0}\geq a^*,$ then
\begin{align}\label{wlb}
    w\geq a^*(1-\eta) \quad in \quad D.
\end{align}
\end{lemma}
\begin{proof}Set
\begin{align}\label{g11-1}
g=w-a^*(1-\eta)+\varepsilon \tau+\varepsilon
\end{align}
  where $\varepsilon$ is any positive constant. Our goal is to prove $g$ is nonnegative in $\bar{D}$ by the maximum principle.

\textit{Step 1} We will prove the minimum of $g$ is only achieved on $\overline{D}\setminus D.$

If there is a  minimum point $z_{min}$ in  $D,$
 then at  $z_{min}$, $\partial_\tau g\leq 0,\,\,\partial_\xi g\leq 0,\,\,w^2\partial_\eta^2 g\geq0$ and thus
 \begin{align}\label{l0g7-1}
 L_0g(z_{min})\geq0,
 \end{align}
 where we recall \eqref{defl0} for $L_0$'s definition.

 However, $$L_0g=-\varepsilon<0$$ in $D$ by \eqref{g11-1}, since $L_0 w=0$. It contradicts to \eqref{l0g7-1}.  Hence, there is no
 minimum point $z_{min}$ in  $D.$
Then the minimum is achieved on $\overline{D}\setminus D.$

\textit{Step 2} We will prove $g\geq 0$ on $\overline{D}\setminus D.$

By the assumption $w|_{\eta=0}\geq a^*,$ we have,
$g|_{\eta=0}>0$ by \eqref{g11-1}. By \eqref{eq:Pran-Cro}, $w|_{\eta=1}=0$ and therefore $g|_{\eta=1}>0$ by \eqref{g11-1}.
By \eqref{g11-1} and \eqref{w0w1a}, $g|_{\xi=0}>0$ and $g|_{\tau=0}>0.$

In summary, $g\geq 0$ in $\overline{D}$. Letting $\varepsilon$ go to $0, $ we have the desired results.

\end{proof}

\begin{lemma}\label{unilbd1}
Let $w$ be the Oleinik monotone solution  of \eqref{eq:Pran-Cro} in $D$ satisfying the assumptions in Theorem \ref{w1tB}.
Then there exists a positive constant $\beta$ depending on   $\mu, C_0$ and $X$  such that
\begin{align}\label{531}
 w(\tau,\xi, \eta)\geq b(1-\eta)\quad in \quad \overline{D},
\end{align} where $4b=c_0e^{-\beta X}e^{
 -\frac{8}{3}}.$
\end{lemma}
\begin{proof} We will use a  bootstrap argument. By the initial condition, $w(0,\xi,0)\geq 4b.$ By Lemma \ref{eta0-1}, if $w|_{\eta=0}\geq2b$, then \eqref{531} holds. Otherwise, there exists $ T^*\in(0,T)$ such that \begin{align} \label{T*7}
    T^*=\sup\{s\in[0,T]|w\geq 2b\quad for \quad (\tau,\xi,\eta)\in[0,s]\times[0,X]\times \{\eta=0\}\}
\end{align} and $w(T^*,\xi_0,0)=2b$ for some $\xi_0\in[0,X].$
 Set
 \begin{align}\label{ds-12}
   D_{T^*}=(0,T^*]\times(0,X]\times(0,1)
\end{align}
and
\begin{align}\label{g11-2}
g=-c_0e^{-\beta\xi}(1-\eta)e^{\frac{8}{3}\eta
 -\frac{8}{3}}+\varepsilon\tau+w\quad in \quad \overline{D_{T^*}}.
\end{align} \textit{ Our goal} is to prove
\begin{align}\label{con-1}
g \geq0\quad  in \quad\overline{D_{T^*}}
\end{align}
by the maximum principle. \textit{After proving this, } letting $\varepsilon\to 0,$ we have
\begin{align*}
w\geq c_0e^{-\beta\xi}(1-\eta)e^{\frac{8}{3}\eta
 -\frac{8}{3}}\quad in \quad \overline{D_{T^*}}.
\end{align*}  Then $w\geq 4b(1-\eta)$ in $ \overline{D_{T^*}}$ and in particular, $w|_{\eta=0}\geq 4b>2b$ in $ [0,T^*]\times[0,X]\times\{\eta=0\},$ which \textit{ contradicts} to the definition of $T^*$ in \eqref{T*7}. Finally, we  close the bootstrap argument and prove
$w\geq 2b(1-\eta)$ in $ \overline{D}$ by Lemma \ref{eta0-1}.

The proof of \eqref{con-1} is divided into two steps.

\textit{Step 1} We will prove
$ g$ does not attain its minimum in $\{0\leq\eta\leq \frac{1}{4}\}\cap \overline{D_{T^*}}.$

 By Lemma \ref{eta0-1}, $w\geq 2b(1-\eta)$ in $D_{T^*}.$ Hence, by
Proposition \ref{eta2negVIPB} and $\partial_\eta w|_{\eta=0}=0$ in \eqref{detaw=0}, we have
\begin{align}\label{eta2n11}
   \partial_\eta^2 w\leq \frac{2\delta}{b\frac{3}{4}}\leq \frac{8\delta}{3b},\,\, \partial_\eta w\leq \frac{2\delta}{3b}\leq \frac{b}{15}\quad in \quad&\{\eta\in[0,\frac{1}{4}] \}\cap\overline{D_{T^*}},
\end{align} where we have used $\delta\leq \frac{b^2}{10}.$
Then for $0\leq\eta\leq \frac{1}{4},$ by \eqref{g11-2} and \eqref{eta2n11} ,
 \begin{align*}
    \partial_\eta g&=-c_0e^{-\beta\xi}\big(-1+(1-\eta)\frac{8}{3}\big)e^{\frac{8}{3}\eta
 -\frac{8}{3}}+\partial_\eta w\\
 &\leq -c_0e^{-\beta\xi}e^{\frac{8}{3}\eta
 -\frac{8}{3}}+ \frac{b}{15}<0.
 \end{align*} Hence, $ g$ does not attain its minimum in $\{0\leq\eta\leq \frac{1}{4}\}\cap \overline{D_{T^*}}.$

 \textit{Step 2} We will prove no minimum of $g$ is  attained in $\{\frac{1}{4}<\eta<1\}\cap D_{T^*}.$

 Recall \eqref{defl0} for $L_0$'s definition. By  Proposition \ref{prop:monoB},  for a constant $C$ depending only on $C_0$ and $\mu,$
\begin{align}\label{dmu11}\begin{split}
    L_0g&\leq [-\eta\beta(1-\eta)+w^2C]c_0e^{\frac{8}{3}\eta-\frac{8}{3} } e^{-\beta \xi}-\varepsilon
    \\&\leq [-\frac{1}{4}\beta(1-\eta)+C^2(1-\eta)^{\frac{3}{2}}]c_0e^{\frac{8}{3}\eta-\frac{8}{3} } e^{-\beta \xi}-\varepsilon\\&<0 \quad in \quad \{\eta> \frac{1}{4}\}\cap D_{T^*},\end{split}
\end{align}by taking $\beta$ large depending on $C$.
Hence,
 no minimum of $g$ is  attained in $\{\frac{1}{4}<\eta<1\}\cap D_{T^*}.$

 In summary, $g$  attains its minimum   \textit{only} on  the boundary $\{\tau=0\}\cup\{\xi=0\}\cup\{\eta=1\}$. Then
$g\geq 0$ in
 $\overline{D_{T^*}}$ since
\begin{align}\label{ibg11}
g \geq0\quad on\quad ( \{\tau=0\}\cup\{\xi=0\}\cup\{\eta=1\})\cap \overline{D_{T^*}}
\end{align} by \eqref{g11-2} and  the initial and boundary data \eqref{ulbwB}.

\end{proof}

\subsection{The class of solutions}
\begin{proof}[Proof of Theorem \ref{w1tB}]

Theorem \ref{w1tB} is a straight consequence of Lemma \ref{unilbd1}, Proposition \ref{eta2negVIPB} and Proposition \ref{prop:monoB}.
\end{proof}

 By Theorem \ref{w1tB},
\begin{align}\label{dyu1-uB}
   b\leq \frac{\partial_y u}{1-u},\quad and \quad \frac{\partial_y u}{(1-u)\sqrt{-\ln(\mu(1-u))}}\leq C_0.
\end{align}  Then $e^{-by}\geq 1-u\geq c e^{- C y^2}$ where $c$ and $C$ are constants depending only on $C_0$ and $\mu$.
Then by \eqref{dyu1-uB},
\begin{align}\label{uyulbdB}
ce^{- C y^2}\leq\partial_y u\leq C e^{-\frac{b}{2}y},
\end{align} \textit{where $c$ and $C$ are constants depending only on $b$, $C_0$ and $\mu$.}
 Hence, for any $y_0\in \R_+,$
\begin{align}\label{1-uy0B}
ce^{-C y_0^2}\leq 1-u(y_0)   =\int_{y_0}^{+\infty}\partial_{y} u dy\leq Ce^{-\frac{b}{2}y_0}
\end{align}  where $c$ and $C$ are constants depending only on $b$, $C_0$ and $\mu$.

\section{stability}
We are ready to prove the stability in Theorem \ref{snearb}.
\subsection{Orbital stability \eqref{orbitalst}}

   \begin{lemma} \label{w2t}Let $w $ and $\bar{w}$ be solutions  to \eqref{eq:Pran-Cro} in $D$ satisfying the  conditions in Theorem Theorem \ref{w1tB} and
   $\partial_\eta^2 \bar{w}\leq 0$ in $ D.$
If for a positive constant $\varepsilon,$  \begin{align}\label{initep}
\|w_1- \bar{w}_1\|_{L^\infty([0,T]\times[0,1])}\leq \varepsilon,\quad  \|w_0- \bar{w}_0\|_{L^\infty([0,X]\times[0,1])}\leq \varepsilon,
\end{align}  then
$$\|w- \bar{w}\|_{L^\infty(D)}\leq \varepsilon.$$
   \end{lemma}

\begin{remark}We only require \textit{one} solution satisfying $\partial_\eta^2 \bar{w}\leq 0$ in $ D.$ For example,  we can use this to  consider  the stability of Blasius solution $w_B=\partial_y u_B$ and $\partial_\eta^2 w$ can take positive values. \end{remark}
\begin{proof}
By  \eqref{detaw=0}, we have $\partial_\eta w|_{\eta=0}=0$.
Set
\begin{align}\label{sp12}
S=w-\bar{w},
\end{align} and
\begin{align}\label{g12}
 g=S+\varepsilon+\mu_1 \tau+\mu_1(1-\eta).
\end{align}

 Then by \eqref{eq:Pran-Cro}, \eqref{sp12} and recalling $J$ in \eqref{J(h)}, we have
\begin{equation}\label{Seq}
\left\{\begin{aligned}
&J(S)=0\quad (\tau,\xi,\eta)\in D,\\
&\partial_\eta S\mid _{\eta=0}=0,\quad \displaystyle\lim_{\eta\to 1} S= 0 .
\end{aligned}\right.
\end{equation} By \eqref{g12}, \eqref{Seq} and $\partial_\eta^2 \bar{w}\leq 0$,
\begin{align}\label{Smu}
    J(g)=-\mu_1+\big((w+\bar{w})\partial_{\eta }^2 \bar{w}\big)(\varepsilon+\mu_1 \tau+\mu_1(1-\eta))\leq -\mu_1<0.
\end{align}

Our goal is to prove that $g$ is nonnegative in $D$ by the maximum principle.

\textit{Step 1} We will prove $g$ does not have a negative minimum on $\overline{D}\setminus D.$

By \eqref{initep} and \eqref{g12},
  $$g\geq 0\quad on \quad \{\tau=0\}\times[0,X]\times[0,1]\cup[0,T]\times\{\xi=0\}\times[0,1].$$

 By \eqref{g12} and \eqref{Seq},  $\displaystyle\lim_{\eta\to 1} g=\varepsilon+\mu_1 \tau>0.$

 By  \eqref{g12} and \eqref{Seq}, $\partial_\eta g\mid _{\eta=0}=-\mu_1<0$ and therefore $g$ does not attain a minimum on $\eta=0.$

  Hence, $g$ does not have a negative minimum on $\overline{D}\setminus D.$

\textit{Step 2} We will prove $g$ does not attain a negative minimum in $  D.$

 If $g$ has a \textit{negative} minimum at some point $p_0 \in  D,$ then $\partial_\tau g\leq 0, \,\, \partial_\xi g\leq 0,\,\, w^2\partial_{\eta }^2 g\geq 0$ at $p_0,$ which implies
  $J(g)\geq 0$ at $p_0 $ by  $\partial_\eta^2 \bar{w}\leq 0$, which contradicts to \eqref{Smu}. Hence,  $g$ does not attain a negative minimum in $  D.$

 In summary, $g\geq0 $ in $\overline{D}.$ Letting $\mu_1$ go to $0,$ we have $S\geq -\varepsilon.$ Letting $S_-=-S$, we can prove $S_-\geq -\varepsilon$ by the same method and therefore $S\leq \varepsilon.$ Then the proof is complete.

\end{proof}

\begin{lemma}\label{s1ulem0} Let $\bar{u}$ and $u$ be Oleinik monotone solutions  in $(0,+\infty)\times(0,X]\times\R_+$ satisfying the conditions in Theorem \ref{w1tB}. Then for any $y_0\in\R_+,$\begin{align}\label{strcore}
|\bar{u}(y)-u(y)|\leq C_{y_0}\int_0^{1}|\bar{w}-w|d\eta
\quad for \quad y\in[0,y_0],
\end{align} where $$C_{y_0}=Ce^{C y_0^2}$$ with $C$ depending only on $\mu, c_0,$ $C_0$ and $X$.
\end{lemma}

\begin{proof}

Let \textit{$y$ be any constant in $(0,y_0]$} and
 $u$ and $w$ be solutions to the original Prandtl equation and the version in Crocco variables.
By the definition of Crocco transformation,  we have $$y=\int_0^{u(y)}\frac{ds}{w(\tau,\xi,s)}.$$
Hence,
$$\int_0^{\bar{u}(y)}\frac{d\eta}{\bar{w}(\tau,\xi,\eta)}=
\int_0^{u(y)}\frac{d\eta}{w(\tau,\xi,\eta)}.$$
Without loss of generality, we can assume
\begin{align}\label{bartildeorder}
u(y)\leq \bar{u}(y).
\end{align}
Then
\begin{align}\label{yeta}
\int_{u(y)}^{\bar{u}(y)}\frac{d\eta}{\bar{w}}-
\int_0^{u(y)}\frac{\bar{w}-w}{\bar{w}w}d\eta=0.
\end{align}

 Since $w\geq b(1-\eta)$ by Theorem \ref{w1tB} and by \eqref{1-uy0B},
$$ce^{-C y_0^2}\leq 1-u(y_0) $$ where $c$ and $C$ are constants depending only on $b$, $C_0$ and $\mu$,
we have $\partial_y u \geq0$ and  for $y\in[0, y_0]$ and $\eta\in[0, u(y)],$
\begin{align}\label{wgamma}
 \bar{w}(\tau,\xi,\eta)\,\, and \,\,w(\tau,\xi,\eta)\geq b(1-\eta)\geq b(1-u(y))
\geq b(1-u(y_0))\geq \gamma_0,
\end{align} where
\begin{align}\label{gamma0}
\gamma_0=bce^{-C y_0^2}
\end{align} with $c$ and $C$  depending only on $b$, $C_0$ and $\mu$.
Then for  $\eta\in[0,u(y)]$ (note $ y\leq y_0$),
\begin{align}\label{doublew}
\bar{w}w\geq \gamma_0^2.
\end{align}
Moreover, since  $w,\,\bar{w}\leq C_0(1-\eta)\sqrt{-\ln(\mu(1-\eta))}\leq C$ with $C$ depending only on $C_0$ and $\mu$ by Theorem \ref{w1tB},
\begin{align}\label{mimw}
   \frac{1}{w},  \,\frac{1}{\bar{w}}\geq  \frac{1}{C},
\end{align} where $C$ is independent of $T$ and $X.$ Then, by \eqref{yeta}, \eqref{mimw} and \eqref{doublew},
\begin{align}
\frac{1}{C}|\bar{u}(y)-u(y)|\leq|\int_{u(y)}^{\bar{u}(y)}\frac{d\eta}{\bar{w}}|\leq
\int_0^{u(y)}\frac{|\bar{w}-w|}{\bar{w}w}d\eta
\leq \frac{\int_0^{1}|\bar{w}-w|d\eta}{\gamma_0^2},
\end{align} where we have used $0\leq u(y)\leq 1$ and recall $\gamma_0$ in \eqref{gamma0}.

\end{proof}

\begin{lemma}\label{s1ulem} Let $\bar{u}$ and $u$ be  Oleinik monotone solutions in $(0,+\infty)\times(0,X]\times\R_+$ satisfying the conditions in Theorem \ref{w1tB} and $\partial_\eta^2 \bar{w}\leq 0$ in $ D.$ Then for any positive constant $\varepsilon,$ there exists a positive constant $\delta_\varepsilon$ such that if \begin{align*}
\|w_1- \bar{w}_1\|_{L^\infty([0,+\infty)\times[0,1])}\leq \delta_\varepsilon,\quad  \|w_0- \bar{w}_0\|_{L^\infty([0,X]\times[0,1])}\leq \delta_\varepsilon,
\end{align*}  then
$\|u- \bar{u}\|_{L^\infty([0,+\infty)\times[0,X]\times[0,+\infty))}\leq \varepsilon.$
\end{lemma}
\begin{proof}
By \eqref{1-uy0B}, for any positive constant $\varepsilon,$ we can take $y_\varepsilon$ big depending on $\varepsilon, b, \mu$ and $C_0 $ such that
\begin{align*}
0\leq 1-u(y_\varepsilon)   \leq \frac{\varepsilon}{2}.
\end{align*} Since $\partial_y u=w \geq0,$ we have, for $y\in[y_\varepsilon,+\infty),$ \begin{align}\label{y0B}|1-u(y)|\leq \frac{\varepsilon}{2}.\end{align}
Then
\begin{align}\label{y0small}
   |\bar{u} -u|\leq |\bar{u} -1|+|1-u|\leq \varepsilon\quad in \quad [0,+\infty)\times[0,X]\times[y_\varepsilon,+\infty).
\end{align}
Next, we consider in the domain $[0,+\infty)\times[0,X]\times[0,y_\varepsilon].$
Let \textit{$y$ be any constant in $(0,y_\varepsilon]$}.
By \eqref{strcore},
\begin{align}
|\bar{u}(y)-u(y)|\leq C_{y_\varepsilon}\int_0^{1}|\bar{w}-w|d\eta,
\end{align} where $C_{y_\varepsilon}$ depends only on $y_\varepsilon$, $\mu, c_0,$ $C_0$ and $X$.
Taking $\delta_\varepsilon\leq \frac{\varepsilon}{C_{y_\varepsilon}},$ by Lemma \ref{w2t},
we have
\begin{align}
   |\bar{u} -u|\leq \varepsilon\quad in \quad [0,+\infty)\times[0,X]\times[0,y_\varepsilon].
\end{align}
Combining with \eqref{y0small}, we have the desired results.

\end{proof}

\begin{proof}[Proof of orbital stability \eqref{orbitalst}]

We will prove, for any positive constant $\varepsilon$, if
\begin{align}\label{difutilde-1}
|\partial_y \bar{u}_1- \partial_y u_1|\leq \delta_0,\quad |\partial_y \bar{u}_0- \partial_y u_0|\leq \delta_0,
\end{align}  hold for sufficiently small $\delta_0,$ then
\begin{align}\label{0710-1}
|\bar{w}_1- w_1|\leq \varepsilon
\end{align} and
\begin{align}\label{0710-2}
|\bar{w}_0- w_0|\leq \varepsilon.
\end{align}
After proving this,   we draw the conclusion  by Lemma \ref{s1ulem}.

Since we use the same method to prove \eqref{0710-1} and \eqref{0710-2}, we only prove \eqref{0710-1}.

To prove \eqref{0710-1} is  to prove: if $
|\partial_y \bar{u}_1- \partial_y u_1|\leq \delta_0$ holds for some sufficiently small positive constant $\delta_0,$ then
for any positive constants $y_2$ and $y_1$ such that
\begin{align}\label{y12eq}
\bar{u}_1(t,y_1)=u_1(t,y_2),
\end{align}
it holds
\begin{align}\label{djdl}
|\partial_{y}\bar{u}_1(t,y_1)-\partial_{y}u_1(t,y_2)|\leq  \varepsilon.
\end{align}

Now we prove \eqref{djdl}.
Without loss of generality, we assume
 \begin{align}\label{y1y2ord}
 y_1\leq y_2.
\end{align}
By the boundary condition in Theorem \ref{snearb} and a similar calculation for \eqref{uyulbdB}, we have
\begin{align}\label{btu1}
 \quad ce^{- C y^2}\leq\partial_y \bar{u}_1,\, \partial_y u_1\leq C e^{-\frac{c_0}{2}y}
\end{align} and
\begin{align}\label{0710-3}
|\partial_{y}^2u_1|\leq C.
\end{align}
\textit{Proof of \eqref{0710-3}.}   In fact, we can see the dependence of $C$ in $|\partial_{y}^2u_1|\leq C$ by the following calculation. By  Theorem \ref{w1tB},
 $$|\partial_{\eta}^2w|=|\frac{\eta\partial_{\xi}w+\partial_{\tau}w}{w^2}|
 \leq C(1-
 \eta)^{\alpha_0-2}$$ for some $C$
 depending only on $C_1,C_0, c_0,\mu,X$. Then by  $\partial_{\eta}w=0$ on $\eta=0$ in \eqref{detaw=0}, $$|\partial_{y}^2u|=|w\partial_{\eta}w|\leq C(1-
 \eta)^{\alpha_0}\sqrt{-\ln({\mu(1-\eta))}},$$ where  $C$ is a positive constant
 depending only on $C_1,C_0, c_0,\mu,X.$

 \textit{Therefore, we complete the proof of \eqref{0710-3}.}

Then by \eqref{y1y2ord} and \eqref{btu1}, there exists a constant $y_\varepsilon$ depending only on $\varepsilon,C_0, c_0,\mu,X$  such that if $y_1\geq y_\varepsilon,$ then
$$|\partial_{y}\bar{u}_1(t,y_1)-\partial_{y}u_1(t,y_2)|\leq  |\partial_{y}\bar{u}_1(t,y_1)|+|\partial_{y}u_1(t,y_2)|\leq \varepsilon.$$ Hence, we only need to consider the case  \begin{align}\label{y1ye}
 y_1\leq y_\varepsilon.
\end{align}
By \eqref{y1ye} and $\partial_{y}\bar{u}_1\geq0,$ $$u_1(t,y_2)=\bar{u}_1(t,y_1)\leq\bar{u}_1(t,y_\varepsilon).$$
 By \eqref{btu1}, there exists a positive constant $\tilde{y}_\varepsilon$ depending only on $\varepsilon,C_0, c_0,\mu,X$ such that $$y_2\leq \tilde{y}_\varepsilon.$$   Then there exists a positive constant $\gamma_\varepsilon$ depending only on $\varepsilon,C_0, c_0,\mu,X$ such that
 \begin{align}\label{gamep}
  \partial_{y}u_1\geq \gamma_\varepsilon\quad in \quad [0,+\infty)\times[0,y_2].
 \end{align} By \eqref{y12eq},
 \begin{align}
\int_0^{y_1}\partial_y\bar{u}_1(t,y')dy'=\bar{u}_1(t,y_1)=u_1(t,y_2)=\int_0^{y_2}\partial_yu_1(t,y')dy'.
\end{align}
 Then by \eqref{y1ye}, \eqref{difutilde-1} and \eqref{gamep}, \begin{align}
\delta_0 y_\varepsilon\geq|\int_0^{y_1}\partial_y\bar{u}_1-\partial_yu_1dy'|=|\int_{y_1}^{y_2}\partial_yu_1dy'
|\geq |y_2-y_1|\gamma_\varepsilon.\end{align} Therefore, for some positive constant $C_\varepsilon$, we have
\begin{align}\label{y1y2dif}
    |y_2-y_1|\leq C_\varepsilon\delta_0.
\end{align}

 Finally, by \eqref{difutilde-1}, \eqref{0710-3} and \eqref{y1y2dif}  ,
\begin{align*}
 |\partial_{y}\bar{u}_1(t,y_1)-\partial_{y}u_1(t,y_2)|\leq&
 |\partial_{y}\bar{u}_1(t,y_1)-\partial_{y}u_1(t,y_1)|
 +|\partial_{y}u_1(t,y_1)-\partial_{y}u_1(t,y_2)|\\
 \leq&\delta_0+C_\varepsilon\delta_0,
\end{align*}
  where $C_\varepsilon$ depends only on $\varepsilon,\alpha_0,C_1,C_0, c_0,\mu,X$. Hence, taking $\delta_0$ small depending only on $\varepsilon,\alpha_0,C_1,C_0, c_0,\mu,X$, we have \eqref{djdl}.

\end{proof}
\subsection{Asymptotic stability \eqref{anyy0asmpst} and \eqref{Raspst}}
\begin{theorem} \label{w2tB}Let $w $ and $\bar{w}$ be  Oleinik monotone solutions  to \eqref{eq:Pran-Cro} in $D$ satisfying the  conditions in Theorem \ref{w1tB} and
   $\partial_\eta^2 \bar{w}\leq 0$ in $ D.$ Assume
      for a positive constant
      $\beta_0$ such that
       \begin{align}\label{betabal0}
       \beta_0<b^2\alpha_0(1-\alpha_0),
\end{align} it holds, for some positive constant $M_0,$
    \begin{align}\label{1-etaln1-eta1}
       |w_1- \bar{w}_1|\leq e^{-\beta_0\tau}M_0(1-\eta)^{\alpha_0} .
    \end{align}
Then
$$|w- \bar{w}|\leq Me^{-\beta_0\tau}(1-\eta)^{\alpha_0}\quad in \quad D,$$ for some positive constant $M$ depending only on $\alpha_0,M_0,C_0$ and $\mu$.

   \end{theorem}

   \begin{proof} By \eqref{1-etaln1-eta1} and   Theorem \ref{w1tB}, we have, for some positive constant $M,$
\begin{align}\label{1-etaln1-eta}
       |w_1- \bar{w}_1|\leq e^{-\beta_0\tau}M(1-\eta)^{\alpha_0},\quad  |w_0- \bar{w}_0|\leq M(1-\eta)^{\alpha_0}.
    \end{align}

Set
\begin{align}\label{sp13}
S=w-\bar{w},
\end{align}
  \begin{align}\label{g13}
 g=S+ e^{-\beta_0\tau}M(1-\eta)^{\alpha_0}+\mu_1 \tau,
 \end{align}
 and
 \begin{align}\label{Jh13}
 J(h)=-\partial_\tau  h-\eta \partial_\xi h+\big((w+\bar{w})\partial_{\eta }^2 \bar{w}\big)h+w^2\partial_{\eta }^2h.
\end{align}
Then by \eqref{Seq}, $\partial_\eta^2 \bar{w}\leq 0$, \eqref{betabal0} and \eqref{g13},
\begin{align}\label{SmuB}\begin{split}
    J(g)=&-\mu_1+\big((w+\bar{w})\partial_{\eta }^2 \bar{w}\big)(e^{-\beta_0\tau}M(1-\eta)^{\alpha_0}+\mu_1 \tau)
    \\&+e^{-\beta_0\tau}M(1-\eta)^{\alpha_0}[-w^2\alpha_0(1-\alpha_0)
    (1-\eta)^{-2}+\beta_0]
    \\ \leq& -\mu_1+e^{-\beta_0\tau}M(1-\eta)^{\alpha_0}[-b^2\alpha_0(1-\alpha_0)
    +\beta_0]<0,\end{split}
\end{align} since  $w\geq b(1-\eta)$ by Theorem \ref{w1tB}.

Our goal is to prove that $g$ is nonnegative in $D$ by the maximum principle.

\textit{Step 1} We will prove $g$ does not have a negative minimum on $\overline{D}\setminus D.$

 By \eqref{g13} and \eqref{1-etaln1-eta},
  $$g\geq 0\quad on \quad \{\tau=0\}\times[0,X]\times[0,1]\cup[0,T]\times\{\xi=0\}\times[0,1].$$

 By \eqref{g13} and \eqref{Seq},  $\displaystyle\lim_{\eta\to 1} g=\mu_1 \tau\geq0.$

By \eqref{g13} and \eqref{Seq}, $\partial_\eta g\mid _{\eta=0}=-\alpha_0e^{-\beta_0\tau}M<0$ and therefore $g$ does not attain a minimum on $\eta=0.$

Hence, $g$ does not have a negative minimum on $\overline{D}\setminus D.$

\textit{Step 2} We will prove $g$ does not have a negative minimum in $  D.$

 If $g$ has a \textit{negative} minimum at some point $p_0 \in  D,$ then $\partial_\tau g\leq 0, \,\, \partial_\xi g\leq 0,\,\, w^2\partial_{\eta }^2 g\geq 0$ at $p_0,$ which implies
  $J(g)\geq 0$ at $p_0 $ by  $\partial_\eta^2 \bar{w}\leq 0$, which contradicts to \eqref{SmuB}. Hence,  $g$ does not have a negative minimum in $  D.$

 In summary, $g\geq0 $ in $\overline{D}.$ Letting $\mu_1$ go to $0,$ we have $S\geq -e^{-\beta_0\tau}M(1-\eta)^{\alpha_0}.$ Letting $S_-=-S$, we can prove $S_-\geq -e^{-\beta_0\tau}M(1-\eta)^{\alpha_0}$ by the same method and therefore $S\leq e^{-\beta_0\tau}M(1-\eta)^{\alpha_0}.$ Then the proof is complete.

\end{proof}
\begin{proof}[Proof of asymptotic stability \eqref{anyy0asmpst} and \eqref{Raspst}]
The proof is divided into two steps.

\textit{Step 1 } We will prove \eqref{anyy0asmpst}, i.e.
\begin{align}\label{My0-1}
\|u-\bar{u}\|_{L^{\infty}([0,X]\times[0,y_0])}\leq e^{-t\beta_0}Ce^{C y_0^2}\quad for  \quad t\in [0,+\infty),
\end{align}
with $C$  depending only on $\alpha_0,C_2,\mu, c_0,$ $C_0$ and $X$.

By \eqref{pyupyub}, we have
\begin{align}
       |w_1- \bar{w}_1|\leq e^{-\beta_0\tau}C_2(1-\eta)^{\alpha_0}.
    \end{align}
Then by Theorem \ref{w2tB}, we have
 \begin{align}\label{w-barwuni}
 |w- \bar{w}|\leq Me^{-\beta_0\tau}(1-\eta)^{\alpha_0}
\leq Me^{-\beta_0\tau}\quad in\quad [0,+\infty)\times[0,X]\times[0,1],
\end{align}
for some positive constant $M$ depending only on $\alpha_0,C_2,C_0$ and $\mu.$

Then,  for any $y_0\in\R_+,$ we have, by Lemma \ref{s1ulem0} and \eqref{w-barwuni},
 \begin{align*}
|\bar{u}(y)-u(y)|\leq& Ce^{C y_0^2}\int_0^{1}|\bar{w}-w|d\eta\\
\leq &MCe^{C y_0^2}e^{-\beta_0t}
\quad for \quad y\in[0,y_0],
\end{align*} with $C$ and $M$ depending only on $\alpha_0,C_2,\mu, c_0,$ $C_0$ and $X$.

\textit{Step 2 } We will prove \eqref{Raspst}, i.e.
\begin{align*}
   \| u-\bar{u}\|_{L^{\infty}([0,X]\times[0,+\infty))}\rightarrow 0, \quad as\quad  t\rightarrow+\infty.
\end{align*}

For any $\varepsilon>0,$ by \eqref{1-uy0B}, we can take $y_\varepsilon$ depending only on $\varepsilon$, $b$, $C_0$ and $\mu$ such that
\begin{align*}
0\leq 1-u(y_\varepsilon)   \leq \frac{\varepsilon}{2}.
\end{align*} Since $\partial_y u=w \geq0,$ we have, for $y\in[y_\varepsilon,+\infty),$ \begin{align*}|1-u(y)|\leq \frac{\varepsilon}{2}.\end{align*}
Then
\begin{align}\label{easpst1}
   |u -\bar{u}|\leq |u -1|+|1-\bar{u}|\leq \varepsilon\quad in \quad [0,+\infty)\times[0,X]\times[y_\varepsilon,+\infty).
\end{align}

Next, by \eqref{My0-1},
\begin{align}
\|u-\bar{u}\|_{L^{\infty}([0,X]\times[0,y_\varepsilon])}\leq e^{-t\beta_0}Ce^{C y_\varepsilon^2}\quad for \quad t\in[0,+\infty)
\end{align}with $C$  depending only on $\alpha_0,C_2,\mu, c_0,$ $C_0$ and $X$. Then there exists a positive constant $T_\varepsilon$ depending only on $\varepsilon,\alpha_0,C_2,\mu, c_0,$ $C_0$ and $X$ such that
\begin{align}\label{easpst2}
\|u-\bar{u}\|_{L^{\infty}([T_\varepsilon,+\infty)\times[0,X]\times[0,y_\varepsilon])}\leq \varepsilon.
\end{align}

Then \eqref{Raspst} is a straight consequence of \eqref{easpst1} and \eqref{easpst2}.

\end{proof}

\section{Acknowledgement}

Y. Guo's research is supported in part by NSF grant no. 2106650. Y. Wang is supported by NSFC under Grant 12001383. Z. Zhang is partially supported
by NSFC under Grant 12171010.

 \end{document}